\documentclass[11pt]{amsart}

 \usepackage[margin=1.2in]{geometry}
\usepackage{amsmath,amssymb,amsthm,amscd, amsxtra, amsfonts,graphicx,fancybox}
 
 \usepackage{comment}
\usepackage{color}
\newcommand{\bea}{\begin{eqnarray}}
\newcommand{\eea}{\end{eqnarray}}
\newcommand{\be}{\begin {equation}}
\newcommand{\ee}{\end{equation}}

\newcommand{\g}{  \mathfrak g}

\newcommand{\hg}{\hat \g}
\newcommand{\Z}{\mathbb Z}

\newcommand{\C}{\mathbb C}

\newtheorem{theorem}{Theorem}[section]
\newtheorem{lemma}[theorem]{Lemma}
\newtheorem{prop}[theorem]{Proposition}
\newtheorem{cor}[theorem]{Corollary}

\theoremstyle{definicija}

\theoremstyle{remark}
\newtheorem{remark}[theorem]{Remark}
 
\newenvironment{nospaceflalign*}
 {\setlength{\abovedisplayskip}{0pt}\setlength{\belowdisplayskip}{0pt}%
  \csname flalign*\endcsname}
 {\csname endflalign*\endcsname\ignorespacesafterend}

\numberwithin{equation}{section}

\begin{document}

\title[]{   A duality between vertex superalgebras $L_{-3/2}(\mathfrak{osp}(1\vert 2)) $      and $\mathcal V^{(2)}$ 
 and    generalization to logarithmic vertex algebras 
  }
\author[]{Dra\v zen  Adamovi\' c}
\author[]{Qing Wang}
  
  \begin{abstract}
  We introduce  a subalgebra $\overline F$ of the Clifford vertex superalgebra ($bc$ system) which is  completely reducible as a $L^{Vir} (-2,0)$--module, $C_2$--cofinite, but it is not conformal and it is not isomorphic to the symplectic fermion algebra $\mathcal{SF}(1)$.   We show that $\mathcal{SF}(1)$ and $\overline{F}$ are in an interesting  duality, since $\overline{F}$ can be equipped with the structure of a $\mathcal{SF}(1)$--module and vice versa. 
  
  Using  the decomposition of $\overline F$  and a free-field realization from \cite{A-2019}, we decompose $L_k(\mathfrak{osp}(1\vert 2))$ at the critical level $k=-3/2$ as a module for $L_k(\mathfrak{sl}(2))$.  The decomposition of $L_k(\mathfrak{osp}(1\vert 2))$   is  exactly the same as  of  the  $N=4$ superconformal vertex algebra with central charge $c=-9$, denoted by $\mathcal V^{(2)}$.  Using the duality between $\overline{F}$ and $\mathcal{SF}(1)$,  we prove that  $L_k(\mathfrak{osp}(1\vert 2))$ and $\mathcal V^{(2)}$ are in the  duality of the same type.  As an application, we construct and classify all irreducible $L_k(\mathfrak{osp}(1\vert 2))$--modules in the category $\mathcal O$ and the category $\mathcal R$ which includes relaxed highest weight modules. 
  We  also describe the structure of  the parafermion algebra $N_{-3/2}(\mathfrak{osp}(1\vert 2))$ as a  $N_{-3/2}(\mathfrak{sl}(2))$--module.
  
 We extend this example, and  for each $p \ge 2$, we introduce a non-conformal vertex algebra $\mathcal A^{(p)}_{new}$    and show that  $\mathcal A^{(p)}_{new} $ is isomorphic to the doublet vertex algebra as a module for the Virasoro algebra. We also construct the vertex algebra $ \mathcal V^{(p)} _{new}$ which is isomorphic to the logarithmic vertex algebra $\mathcal V^{(p)}$ as  a module for $\widehat{\mathfrak{sl}}(2)$.

  \end{abstract}
\maketitle
\section{Introduction}

In the representation theory of vertex algebras a special emphasis was put on $C_2$--cofinite, non-rational vertex algebras. Recently these vertex algebras are also called logarithmic vertex algebras since they appeared in logarithmic conformal field theory (cf. \cite{CR-2013}, \cite{CG}). Basic examples are  the triplet vertex algebras $\mathcal W^{(p)}$ (cf.  \cite{AM08}, \cite{FGST}, \cite{FGST2}) and the symplectic fermion vertex superalgebra $\mathcal{SF}(d)$ (cf. \cite{Abe}, \cite{Ka00}).
The triplet vertex algebra $\mathcal W^{(p)}$ has a simple current extension $\mathcal A^{(p)}$ (cf. \cite{AM-doublet}), which is called the doublet vertex algebra. Note that for $p= 2$, $\mathcal A^{(2)} \cong \mathcal SF(1)$. 

 Dualities like the Kazama-Suzuki duality and the  S-duality, have recently appeared  in the  papers  on  logarithmic vertex algebras, affine and $\mathcal W$--algebras (cf.  \cite{AK-2021},   \cite{CGL}, \cite{CGN}, \cite{FG}). In the current   paper, we    begin  a study of a different duality which relates the logarithmic vertex algebras with some non-conformal vertex algebras, like    affine  vertex superalgebras  at the critical level.

 Let $F$ be a Clifford vertex superalgebra generated by the charged fermionic fields $\Psi^{\pm}$ with non-trivial $\lambda$-bracket $[ \Psi ^+ _{\lambda} \Psi^-] = 1$.
Then  $\mathcal{SF}(1)$ is realized as:
$$ \mathcal{SF}(1) =\mbox{Ker}_F \int \Psi^-(z) dz. $$
The even part of $\mathcal SF(1)$   is isomorphic to the triplet vertex algebra $ \mathcal SF(1) ^+ \cong \mathcal W^{(p)}$ with  $p=2$ (cf. \cite{Abe}, \cite{AM08}).

The triplet vertex algebras and the symplectic fermion of rank one $\mathcal{SF}(1)$ are completely reducible modules for the Virasoro algebra. Let $L^{Vir}(c,h)$ denotes the irreducible, highest weight module for the Virasoro algebra of central charge $c$ and highest weight $h$. Then

\bea \mathcal{SF}(1) = \bigoplus_{n=0} ^{\infty} (n+1) L^{Vir}(-2, \frac{n ^2 +n}{2}). \label{dec-sf-uvod} \eea

A $\widehat{\mathfrak{sl}}(2)$-version of logarithmic vertex algebras were introduced in \cite{A-TG} and studied also in \cite{ACGY}. These algebras are called the $\mathcal V^{(p)}$--algebras. The $\mathcal V^{(p)}$ algebras are completely reducible modules for the affine Lie algebra $\widehat{\mathfrak{sl}}(2)$ at level $-2 + \frac{1}{p}$.
It was  proved that the quantum Hamiltonian reduction functor $H_{DS}$ sends  $\mathcal V^{(p)}$ to $\mathcal A^{(p)}$.

A very interesting case is  $p=2$  when  the vertex superalgebra $\mathcal V^{(2)}$  carries the structure of the small $N=4$ superconformal algebra with central charge $c_{1/2}=-9$ (cf. \cite{A-TG}). This vertex superalgebra has also appeared in   the four-dimensional super Yang-Mills theory in physics (cf. \cite{BMR}).

  Recall that the small $N=4$ superconformal algebra at central charge
$c_k= - 6(k +1)$ is 
 realised as the  minimal  affine $W$--algebra
$ \mathcal W_{k}(\mathfrak{psl}(2\vert2), f_{\theta})$ (cf. \cite{KW}). So in our case 
$\mathcal V^{(2)} = \mathcal W_{k}(\mathfrak{psl}(2\vert2), f_{\theta})$ with $k=1/2$.
Then we have:

\bea \mathcal V^{(2)} = \bigoplus_{n=0}  ^{\infty}(n+1) L^{\mathfrak{sl}(2)} _{-3/2}(n \omega_1). \label{dec-n4-uvod} \eea

 In the current paper we shall investigate  two non-conformal vertex algebras $\overline F$ and $L_{-3/2}(\mathfrak{osp}(1\vert 2))$ such that
\bea \overline F &=& \bigoplus_{n=0} ^{\infty} (n+1) L^{Vir}(-2, \frac{n ^2 +n}{2}), \label{dec-f-new-uvod} \\
  L_{-3/2}(\mathfrak{osp}(1\vert 2)) &=& \bigoplus_{n=0} ^{\infty} (n+1) L^{\mathfrak{sl}(2)}_{-3/2}(n \omega_1). \label{dec-osp12-uvod} \eea

 $L_{-3/2}(\mathfrak{osp}(1\vert 2))$ is the unique graded simple quotient of the universal affine vertex superalgebra   $V^{-3/2}(\mathfrak{osp}(1\vert 2))$ at the critical level.
  This vertex superalgebra is not conformal, but it has a vertex subalgebra isomorphic to   $L_{-3/2}(\mathfrak{sl}(2))$.   Thus,  we solve  the branching  rule problem for the embedding $\mathfrak{sl}(2) \hookrightarrow \mathfrak{osp}(1\vert 2)$ at level $k=-3/2$. Quite surprisingly,  we get the decomposition (\ref{dec-osp12-uvod}) which coincides with  (\ref{dec-n4-uvod}). Note that the Sugawara Virasoro vector of $L_{-3/2}(\mathfrak{sl}(2))$ does not define the structure of a vertex operator superalgebra on $L_{-3/2}(\mathfrak{osp}(1\vert 2))$. 
 
 The realization of  $L_k (\mathfrak{osp}(1\vert 2))$  was presented in \cite{A-2019}, so that  $L_{k}(\mathfrak{osp}(1\vert 2)) \hookrightarrow L^{ns}_{c_k} \otimes \Pi(0) ^{1/2}$, where $L^{ns}_{c_k}$ is the simple N=1 Neveu-Schwarz  vertex   superalgebra, and $\Pi(0) ^{1/2}$ is a half-lattice vertex algebra. A critical version of the realization was also presented in  \cite{A-2019}. 
 
  In the current paper, we investigate the critical level case in more details. We get a homomorphism $f : V^{-3/2}  (\mathfrak{osp}(1\vert 2))\rightarrow \overline F \otimes \Pi(0) ^{1/2}$. The vertex algebra   $\overline F$ is   realized as
  $$\overline F =\mbox{Ker}_F \int \Phi_2 (z) dz,  $$
  where $ \Phi_2 (z) = \frac{1}{\sqrt{2}} (\Psi^+ (z) - \Psi^- (z))$ is the neutral fermionic field. We show that $Im (f)$ is the  simple vertex superalgebra:
  $$ L_{-3/2} (\mathfrak{osp}(1 \vert 2)) = \mbox{Ker}_{ \overline F \otimes \Pi(0) ^{1/2}}  S^{osp}  $$
  where $$S^{\mathfrak{osp}} = \int   e^{\frac{\alpha}{2} + \nu} (z)dz :   \overline F \otimes \Pi(0) ^{1/2}  \rightarrow 
     F_{tw} \otimes \Pi_{\nu}$$ is a screening operator defined in Subsection    \ref{screening-real}.
     
      We prove the following results which can explain the coincidences  which we have noticed above:
     
          \begin{itemize}
     \item The vertex superalgebra $\overline{F}$ has the structure of an irreducible 
     $\mathcal{SF}(1)$--module and vice versa (cf. Proposition  \ref{isomorphism-fields-1}).
     
     \item The vertex superalgeba $L_{-3/2}(\g)$ has the structure of an irreducible 
     $\mathcal V^{(2)}$--module and vice versa (cf. Theorem \ref{isom-osp-N4-direct}). 
     \end{itemize}

\begin{table}[h!]
  \begin{center}
    \caption{ Correspondences between irreducible $\mathcal V^{(2)}$--modules and $L_{-3/2}(osp(1 \vert 2))$--modules (see Sect. \ref{corresp-modules})}
    \label{tab:table1}
    \begin{tabular}{l|c|c}
      \mbox{Vertex algebra} &    $\mathcal V^{(2)}$ &    $L_{-3/2}(\mathfrak{osp}(1 \vert 2))$ \\
      \hline
      \mbox{ordinary modules}  &  $\mathcal V^{(2)}$ &    $L_{-3/2}(\mathfrak{osp}(1 \vert 2))$ \\
      \hline 
       \mbox{category} $\mathcal{O}$ &  $\mathcal V^{(2)}$  & $L_{-3/2}(\mathfrak{osp}(1 \vert 2))$  \\ 
       & $L^{N=4} (U_{-1})$  & $L_{-3/2} ^{\mathfrak{osp}(1 \vert 2)} (-\omega_1)$ \\
      \hline
      \mbox{Indecmposable modules in } \ $\mathcal{O}$ & $\mathcal M_{-1} ^{N=4}$ & $\mathcal M_{-1} ^{\mathfrak{osp}(1 \vert 2) }$ \\ \hline
      \mbox{Relaxed h.w. modules}  { } &  \\ 
      &  $L^{N=4} (U_{-1, r})$    & $L_{-3/2} ^{\mathfrak{osp}(1 \vert 2)} (U_{0}(r))$   
    \end{tabular}
  \end{center}
\end{table}
It is important to note that modules in the same lines are isomorphic as $L_{-3/2}(\mathfrak{sl}(2))$--modules.
     
         In this article, we show that this interesting connection can be extended on a broader class of logarithmic vertex algebras.  For each $p \ge 2$, in Section \ref{generalizacija} we construct non-conformal vertex algebras $\mathcal A^{(p)}_{new}$ and $\mathcal V^{(p)} _{new}$  and show that
         \begin{itemize}         
\item  $\mathcal A^{(p)}_{new} \cong \mathcal A^{(p)}$;   $\mathcal W^{(p)}_{new} \cong \mathcal W^{(p)}$ as  modules for the Virasoro algebra;
\item $\mathcal V^{(p)}_{new} \cong \mathcal V^{(p)}$  as a modules for the affine Lie algebra $\widehat{\mathfrak{sl}}(2)$.
\end{itemize}
The proof is based on the construction of an explicit Virasoro (resp. $\widehat{\mathfrak{sl}}(2)$)--isomorphism.

\section{Clifford vertex superalgebra $F$ and its twisted modules}

We assume that the reader is familiar with basic concepts in the theory of vertex algebras and  the representation theory of  affine Lie algebras and the Virasoro algebras.
  Let $(V,Y, {\bf 1})$ be a vertex superalgebra (cf. \cite{K2},  \cite{LL}). The derivation in the vertex superalgebra $V$ is denoted by $D$. Let  $L^{Vir}(c,0)$  be the simple Virasoro vertex algebra of central charge $c$.
 Let $L^{Vir}(c,h)$  denotes the irreducible, highest weight module for the Virasoro algebra of central charge $c$ and highest weight $h$.

\subsection{Vertex superalgebra $F$ and its Virasoro vectors} The Clifford vertex superalgebra $F$ is the universal vertex superalgebra  generated by the odd fields  $\Phi_i $, $i=1,2$, and the following $\lambda$--brackets: $$ [  (\Phi _ i )_{\lambda} \Phi_j] =  \delta_{i,j}. $$ 

The vertex superalgebra $F$ has the structure of an irreducible module for the Clifford algebra $\textit{Cl}(A)$ associated to the vector superspace $A =\mathbb{C}\Phi_1 \oplus \mathbb{C}\Phi_2$ with generators 
 $\Phi_i  (r),  \ \ r \in \tfrac{1}{2} + {\Z}$  and anti-commutation relations
$$[\Phi_i  (r), \Phi_j  (s)  ]_+ = \delta_{r+s,0} \delta_{i,j}.$$

As a vector space
$F   = \bigwedge \left( \Phi_i (1/2-n) \ \vert \    \ n  \in {\Z}_{>0}, i=1,2  \right)$.

 Let $$\Psi^+ = \frac{1}{\sqrt{2}} (\Phi_1 + \sqrt{-1} \Phi_2), \quad \Psi^- = \frac{1}{\sqrt{2}} (\Phi_1 - \sqrt{-1} \Phi_2). $$
 Then
 $$ \Phi_1 = \frac{1}{\sqrt{2}} (\Psi^+ + \Psi^-), \quad  \Phi_2 = \frac{1}{\sqrt{-2}} (\Psi^+ -  \Psi^-). $$

By the boson-fermion correspondence we have that $F\cong V_{\Z \alpha}$, where $\langle \alpha, \alpha \rangle = 1$ and
$ \Psi^{\pm} = e^{\pm \alpha}. $
Therefore
$$   \Phi_1 = \frac{1}{\sqrt{2}} (e^{\alpha}  + e^{-\alpha}), \quad  \Phi_2 = \frac{1}{\sqrt{-2}} (e^{\alpha}  - e^{-\alpha}). $$
 \bea  \omega_{sf} &=& D\Psi^+ \Psi^- \nonumber \\ &=&   \frac{1}{2} \left ( 
 \Phi_1 (-\tfrac{3}{2} )   \Phi_1 (-\tfrac{1}{2})  +  \Phi_2 (-\tfrac{3}{2} )   \Phi_2 (-\tfrac{1}{2}) \right) \nonumber  \\
 && +  \frac{ \sqrt{-1}}{2} \left ( 
 \Phi_2 (-\tfrac{3}{2} )   \Phi_1 (-\tfrac{1}{2})  -  \Phi_1 (-\tfrac{3}{2} )   \Phi_2 (-\tfrac{1}{2}) \right) \nonumber \\
 D\Psi^+ \Psi^+ &=&  \frac{1}{2} \left ( 
 \Phi_1 (-\tfrac{3}{2} )   \Phi_1 (-\tfrac{1}{2})  -  \Phi_2 (-\tfrac{3}{2} )   \Phi_2 (-\tfrac{1}{2}) \right) \nonumber  \\
 && +  \frac{ \sqrt{-1}}{2} \left ( 
 \Phi_2 (-\tfrac{3}{2} )   \Phi_1 (-\tfrac{1}{2})  +  \Phi_1 (-\tfrac{3}{2} )   \Phi_2 (-\tfrac{1}{2}) \right) \nonumber 
 \eea

Note  that $  \omega_{sf}$ is a Virasoro vector of central charge $c=-2$ in $F$,
$ D\Psi^+ \Psi^+$ is a commutative vector.

  Let $  L_{sf}(n) =  ( \omega_{sf})_{n+1}$.
 Then 
 \bea  \omega&=&   \omega_{sf} +  D\Psi^+ \Psi^+ \nonumber \\
 & =&  \sqrt{-1} \Phi_2(-\tfrac{3}{2}) \Phi_1 (-\tfrac{1}{2} ){\bf 1} + \Phi_1(-\tfrac{3}{2}) \Phi_1 (-\tfrac{1}{2}){\bf 1} \nonumber \\
 &=& \frac{1}{2} \alpha(-1)^2 + \frac{1}{2} \alpha(-2) + e^{2\alpha} \label{def-vir-1} \eea
 is a Virasoro vector in $ F \cong V_{\Z \alpha}$ of central charge $c_{1,2}=-2$.   
 Let $L(n) = \omega_{n+1}$.

\subsection{Symplectic fermion vertex superalgebra $\mathcal{SF}(1)$}\label{s-fermion-uvod}

The symplectic fermion vertex algebra  $\mathcal{SF}(1)$ is defined  as
$$ \mathcal{SF}(1) =\mbox{Ker}_F \int \Psi^-(z) dz. $$
As a vector space, 
$$ \mathcal{SF}(1)  = \bigwedge \left ( \Psi^-  (-r+\tfrac{1}{2}),\Psi^+ (-r-\tfrac{1}{2}) \ \vert \  r \in  {\Z}_{> 0} \right). $$
The vertex superalgebra $ \mathcal{SF}(1) $  is a simple,  $C_2$--cofinite ${\Z}_{\ge 0}$--graded vertex operator superalgebra with conformal vector $\omega_{sf}$ (cf. \cite{Abe}). It is freely generated by the fields
$a^+ = \Psi^- ,  a^- = D\Psi^+$ of conformal weights one.  Set
$ a^+ (n) = \Psi^-(n+1/2), \quad a^-(n) = -n  \Psi^+(n-1/2)$.
Then we have
\bea 
[a^{\pm} (n), a^{\pm}(m)]_+ = 0, \quad [a^{+} (n), a^{-}(m)]_+= n \delta_{n+m,0}. \label{com-sym-ferm}
\eea

The PBW basis of  $\mathcal{SF}(1)$ is
 given by
$$ a^-(-m_1) \cdots a^-  (-m_r)  a^+ (-n_1) \cdots a^- (-n_s) {\bf 1}$$ where $r, s \in {\Z}_{\ge 0}$,  $m_i, n_i \in {\Z}_{>0}$,
$ m_1 > \cdots > m_r \ge 1, \  n_1 > \cdots > n_s \ge 1$.
 
 In other words $\mathcal{SF}(1)$ is the universal affine vertex superalgebra  of level $1$ associated to the Lie superalgebra $\mathfrak{psl}(1\vert 1)$, i.e.
 $$ \mathcal{SF}(1) = L_1( \mathfrak{psl}(1\vert 1)) = V^1 (\mathfrak{psl}(1\vert 1)). $$
 So in order to construct a $\mathcal{SF}(1)$--module, it is enough to construct a restricted $\widehat{\mathfrak{psl}}(1\vert 1)$--module of level $1$.

\subsection{Twisted module $F^{tw}$} Let $\Theta$  be the automorphism of the vertex superalgebra $F$ such that 
$$ \Phi_i \mapsto -\Phi_i, i =1,2. $$
Then a $\Theta$--twisted $F$--module is realised as 
$$(F^{tw}, Y_{tw}):= (F, Y(\Delta(\alpha/2, z)\cdot, z)). $$
$F^{tw}$ is realized as $V_{\Z \alpha + \alpha /2}$.
Note that $v_{tw}=e^{\tfrac{\alpha}{2}}$ is a singular vector for the Virasoro algebra such that
$$U(Vir). v_{tw} \cong L^{Vir}(-2, -\frac{1}{8}). $$

 \section{Vertex superalgebra $\overline F$ and its duality   with $\mathcal{SF}(1)$ }
 \label{new-simpl-ferm}
 
 In this section we introduce the  subalgebra $\overline F$ of $F$ with a Virasoro vector $\omega$ of central charge $c_{1,2} =-2$ defined by (\ref{def-vir-1}).  $\overline F$ is not a vertex operator superalgebra since $L(-1) \ne D$, but it shares similar properties as symplectic fermion vertex superalgebra $\mathcal{SF}(1)$. In particular, we will show that $\overline F$ is isomorphic to $\mathcal{SF}(1)$ as a module for the Virasoro algebra. This case will be a motivated example for introducing  non-conformal duals of logarithmic vertex algebras in Section \ref{generalizacija}.
 
 \vskip 5mm

   Let $G= \Phi_2(\tfrac{1}{2})$.  Define the following vertex subalgebra of $F$: 
 \bea \overline F = \mbox{Ker}_F G= \bigwedge \left ( \Phi_1 (-r+\tfrac{1}{2}),\Phi_2 (-r-\tfrac{1}{2}) \ \vert \  r \in  {\Z}_{> 0} \right). \label{basis} \eea

 \begin{prop} We have:
 \begin{itemize}
 \item[(1)] $ \overline F$  is a simple vertex superalgebra.
 \item[(2)] $ \overline F$  is  strongly generated by $\Phi_1(z)$ and $\partial_z \Phi_2(z)$
 \item[(3)] $ \overline F$ is $C_2$--cofinite, non-rational vertex superalgebra.
 \end{itemize}
 \end{prop}
  \begin{proof}
  The assertions (1) and (2) are clear. Using the basis (\ref{basis}) of $\overline{F}$ we get
  $$ \overline{F}/ C_2(\overline{F}) =\mbox{span}_{\C}\{ \overline{\bf 1}, \overline{\Phi_1}, \overline{D\Phi_2}, \overline{ : \Phi_1 D \Phi_2:} 
\}, $$
  where  for $v \in \overline{F}$, we write $\overline{v} = v + C_2(\overline{F}) \in  \overline{F}/ C_2(\overline{F})$.
  
Finally, as in the case of symplectic fermion, we see that  $F$ is indecomposable, but reducible $\overline{F}$--module. Therefore $\overline{F}$ is non-rational vertex superalgebra.
  \end{proof}
   
Recall   that $ \omega_{sf}$ is a Virasoro vector of central charge $c=-2$ in $F$,
$ D\Psi^+ \Psi^+ = e^{2\alpha}$ is a commutative vector.
 Then 
 $$ \omega= \omega_{sf} +  D\Psi^+ \Psi^+ = \sqrt{-1} \Phi_2(-\tfrac{3}{2}) \Phi_1 (-\tfrac{1}{2} ){\bf 1} + \Phi_1(-\tfrac{3}{2}) \Phi_1 (-\tfrac{1}{2}){\bf 1} $$
 is a Virasoro vector in $\overline F$ of central charge $c=-2$.

 The conformal vector $\omega$ does not define on  $\overline F$ the structure of a vertex operator superalgebra since the derivation $D =  L_{sf}(-1)$ on $\overline{F}$ is not  $L(-1)$. Moreover, 
 $Q =  L(-1) - D=  (D\Psi^+ \Psi^+ )_0$ is a well defined operator   which commutes with the action  of the  Virasoro algebra    $L(n) = L_{sf}(n) + (D\Psi^+ \Psi^+)_{n+1}$,  since it commutes with $L_{sf}(n)$ (cf. \cite{AM08}) and obivously with $(D\Psi^+ \Psi^+)_{n+1}$.
  
  The proof of the following lemma is clear.
  \begin{prop} We have:
  \begin{itemize}
  \item $ [Q, G] = 0$.
  \item $Q$ is a derivation on the vertex superalgebra  $\overline F$ such that
  $[Q, L(n)] =0$ for $n \in {\Z}$.
  \end{itemize}
  \end{prop}

 Define the following operators
\bea  \varphi^- (n) &=& a^-(n) =-\frac{n}{\sqrt{2}}\left(   \Phi_1(  n - \tfrac{ 1 }{2}) + \sqrt{-1}   \Phi_2(n - \tfrac{ 1 }{2}) \right). \nonumber \\
 \varphi^+ (n) &=&  -\sqrt{2}  ( n   \Phi_1( n + \tfrac{ 1 }{2}) +  \sqrt{-1} (n+1)  \Phi_2( n + \tfrac{ 1 }{2}) )  \nonumber  \\
 &=&-(2n +1 ) \Psi ^+ ( n + \tfrac{ 1 }{2}) +     \Psi ^- ( n + \tfrac{ 1 }{2}).  \nonumber 
 \eea
 Note that
 $$ \omega = \varphi^+ (-1) \varphi^- (-1) {\bf 1}. $$
 
 Let $\varphi^{\pm} (z) = \sum_{n \in {\Z} } \varphi^{\pm} (n) z^{-n-1}$.
 Then
 
 \bea
 \varphi^+ (z) &=&   z \sqrt{2} \left(  \partial_z \Phi_1(z)  +  \sqrt{-1} \partial_z \Phi_2(z) \right) + \sqrt{2} \Phi_1(z) \label{embed-1} \\
 \varphi^- (z) &=&  \frac{\sqrt{2}}{2} \left(\partial _z \Phi_1(z) +\sqrt{-1}  \partial _z \Phi_2(z) \right) \label{embed-2}
 \eea
 This implies:
 \bea
 \Phi_1(z) &=&\frac{1}{\sqrt{2}} (   \varphi^+ (z)- 2 z \varphi^-(z)    ) \label{inverse-1}\\
 \partial_z \Phi_2(z) &=&-\sqrt{-2} \varphi^-(z) -\frac{\sqrt{-1}}{\sqrt{2}} \partial_z (  \varphi^+ (z) -  2 z \varphi^-(z) )  . \label{inverse-2}
 \eea
 By a direct calculation we have:
 \begin{lemma} \label{comm-vir} We have:
  $$[L(n),  \varphi^{\pm} (m)  ] =- m \varphi^{\pm} (n+m), \  [\varphi^{\pm}  (m) ,  \varphi^{\pm} (n)  ]_{+}=0, \ [\varphi^{+} (n),\varphi^{-} (m)]_{+} = n \delta_{n+m,0}. $$
 \end{lemma}
 
 \begin{proof} By a direct calculation we get:
 \bea
 [L(n),  \Phi_1 (m+\tfrac{1}{2})] &=&- (n+1 + 2m) \Phi_1 ( n+m+\tfrac{1}{2})
 - \sqrt{-1}  (n+m +1)  \Phi_2( n+m+\tfrac{1}{2})  
  \nonumber \\
  \  [ L(n),  \Phi_2 (m +\tfrac{1}{2})]&=&  -m \sqrt{-1} \Phi_1 (n+m+\tfrac{1}{2}) \nonumber 
 \eea
 This implies  
\bea  [L(n),  \varphi^-(m)  ] &=& - m  \varphi^- (n+m), \nonumber \\
 \  [L(n ),  \varphi^+ ( m)  ] &=&  \sqrt{2} m (- (n+1 + 2m) \Phi_1 ( n+m+\tfrac{1}{2})
 - \sqrt{-1}  (n+m +1)  \Phi_2( n+m+\tfrac{1}{2}) ) \nonumber \\
 && + \sqrt{2} m (m+1)  \Phi_1 ( n+m+\tfrac{1}{2})   \nonumber  \\
 &=& - m \sqrt{2} (   (n+m)    \Phi_1 ( n+m+\tfrac{1}{2})  + \sqrt{-1}  (n+m + 1)   \Phi_2( n+m+\tfrac{1}{2}) ) \nonumber \\
 &=& - m \varphi^+ (n+m). \nonumber 
 \eea
 \end{proof}

Let $\widetilde{\mathcal{SF}}(1)$ be the vertex superalgebra generated by local fields $\varphi^+(z), \varphi^-(z)$ acting on $\overline{F}$.

 \begin{prop} \label{isomorphism-fields-1}We have:
\item[(1)]   $\widetilde{\mathcal{SF}}(1)$ is isomorphic to $\mathcal{SF}(1)$.

\item[(2)] The vertex superalgebra $\overline F$ has the structure of  an irreducible $\mathcal{SF}(1)$--module, denoted by $ (\overline{F}, \mathcal Y^{\mathcal{SF}(1)} _{\overline F} (\cdot, z) )$, which is  uniquely determined by
$  \mathcal Y^{\mathcal{SF}(1)} _{\overline F} (a^{\pm},  z)  = \varphi^{\pm}(z). $
 The action of the Virasoro field is $ \mathcal Y^{\mathcal{SF}(1)} _{\overline F} (\omega_{sf},  z)  = L(z)$.

\item[(3)] The vertex superalgebra $\mathcal{SF}(1)$ has the structure of  an irreducible $\overline F$--module, denoted by $ (\mathcal{SF}(1), \mathcal Y_{\mathcal{SF}(1)} ^{\overline F} (\cdot, z) )$,
determined by formulas (\ref{inverse-1})-(\ref{inverse-2}).
\end{prop}
\begin{proof}
Lemma \ref{comm-vir} together with the fact that $\mathcal{SF}(1) = V^1 (\mathfrak{psl}(1\vert 1))$ (cf. Subsection \ref{s-fermion-uvod}) implies that $\widetilde{\mathcal{SF}}(1)\cong \mathcal{SF}(1)$. So (1) holds. Then using the theory of local fields from  \cite{LL} we get 
map $ a^{\pm} (z) \mapsto \varphi^{\pm} (z) $ which uniquely  define  on $\overline{F}$ the structure of a  $\mathcal{SF}(1)$--module. Assume that   $\overline{F}$  is reducible    $\mathcal{SF}(1)$--module, with a submodule $U \ne \overline{F}$. Then the relations (\ref{inverse-1})-(\ref{inverse-2}) show that $U$ is also a $\overline{F}$--submodule. But this is not possible since $\overline{F}$ is simple vertex superalgebra. This proves the assertion (2). The proof of (3) is completely analogous.
\end{proof}

Since $\overline F\cong \mathcal{SF}(1)$ as a $\mathcal{SF}(1)$--module,  and since   the Virasoro field $L_{sf}(z)$  acts on $\overline{F}$ as  $L(z)$, the results from  \cite{Abe} and \cite{AM08} directly  imply:
\begin{theorem} \label{sing-vekt-1}For each $n \in {\Z}_{>1}$ and $0 \le j \le n $ we have:
\begin{itemize}
\item[(1)] $w_{n,j}:= Q^j  \varphi^+ (-n) \cdots \varphi^+ (-1) {\bf 1}$
is a singular vector in $\overline F$.
\item[(2)] $ Q^{n}   \varphi^+ (-n) \cdots \varphi^+  (-1) {\bf 1} = \nu  \varphi^-  (-n) \cdots \varphi^- (-1) {\bf 1}$ for certain $\nu \ne 0$.
\end{itemize}
 As a $L^{Vir}(-2, 0)$--module:
\bea \overline F = \bigoplus_{n =0}  ^{\infty} (n+1) L^{Vir}(-2, \tfrac{n ^2 + n}{2}). \label{dec-formula} \eea
\end{theorem}

The vertex superalgebra $\overline F$ has the canonical parity automorphism $\sigma$ such that
$$ \Phi_i \mapsto - \Phi_i, i =1,2.$$
Let
$$\overline F^+ = \{ v \in \overline F \ \vert \ \sigma(v) = v\}. $$

\begin{remark}Note that $\overline F$ is not isomorphic to the symplectic fermion vertex superalgebra $\mathcal{SF}(1)$, although these two vertex superalgebras are isomorphic as $L^{Vir}(-2,0)$--modules. \end{remark}

In Section \ref{generalizacija} we shall construct an explicit $L^{Vir}(-2,0)$--isomorphism between
symplectic fermion vertex superalgebra   $ \mathcal{SF}(1)$ and $\overline F. $
Note that $\mathcal{SF}(1) ^+$  is isomorphic to the triplet vertex algebra $\mathcal W^{(p)}$  for $p=2$.

The following result are special case of Theorem  \ref{conj-1} 
in the case $p=2$. 
\begin{prop} The mappings 
\bea
\Omega= \exp[e^{2\alpha} _1] \vert_{ \mathcal{SF}(1)} :  && \mathcal{SF}(1) \rightarrow \overline F, \nonumber \\
\Omega= \exp[e^{2\alpha} _1] \vert_{ \mathcal W^{(2)}  }:  && \mathcal W^{(2)} \rightarrow \overline F^+, \nonumber
 \eea
are  $L^{Vir}(-2,0)$--isomorphisms. Moreover, as operators  of $\mbox{End}(F)$ we have
\bea
\Omega   L_{sf}(n) =  L(n) \Omega,  \quad
\Omega a^{\pm} (n) =      \varphi^{\pm}(n)  \Omega. \nonumber 
  \eea
\end{prop}


\section{Realization of $L_{-\tfrac{3}{2}}(\mathfrak{osp}(1\vert 2))$ and its consequences}
 
 In this section we first recall a realisation from  \cite{A-2019}  which gives a homomorphism  $\Phi : V^{k}(\mathfrak{osp}(1\vert 2)) \rightarrow  \overline F \otimes \Pi(0)^{1/2}$ at $k=-3/2$.  We prove that the image of this homomorphism is the simple vertex superalgebra $L_k(\mathfrak{osp}(1\vert 2))$.   We also construct the screening operator $S^{\mathfrak{osp}}$ for this realisation. By using the decomposition of $\overline F$ as a $L^{Vir}(-2,0)$--module from Section  \ref{new-simpl-ferm} we obtain the decomposition of  $L_k(\mathfrak{osp}(1\vert 2))$ as a $L_k(\mathfrak{sl}(2))$--module.

\subsection{ Affine vertex superalgebra $V^k(\mathfrak{osp}(1 \vert 2)) $}  

Recall that $\g = \mathfrak{osp}(1 \vert 2)$ is the simple complex Lie superalgebra   with basis $\{e,f, h, x, y\}$ such that the even part
$\g_{0} = \mbox{span}_{\C} \{ e, f, h \} \cong \mathfrak{sl}(2) $ and the  odd part $\g_{1}  =\mbox{span} _{\C} \{x, y \}$.
The anti-commutation relations are given by
\bea
 && [e,f] = h, \ [h, e] = 2 e, \ [h, f] = - 2 f \nonumber \\
 && [h, x] = x, \ [ e, x] = 0,  \ [f, x] = -y \nonumber \\
 && [h, y] = -y, \ [e, y] =-x \ [f, y] = 0 \nonumber \\
 && \{ x, x\} = 2 e, \ \{x, y\} = h, \ \{ y, y \} = -2 f. \nonumber
\eea
Choose the non-degenerate super-symmetric bilinear form  $ ( \cdot , \cdot )$ on $\g$ such that non-trivial products are given by
$$ (e, f) = (f,e) = 1, \ (h,h) = 2, \ (x, y) =  -(y, x) = 2. $$
Let ${\widetilde  \g} = {\g} \otimes {\C}[t, t^{-1}] + {\C} K + {\C}\dot{d}  $ be the associated affine Kac-Moody  Lie superalgebra,  where $\dot{d}$ is  the degree operator and $K$ is the  central element. Let $\widehat{\g} = [{\widetilde  \g}, {\widetilde  \g}] = {\g} \otimes {\C}[t, t^{-1}] + {\C} K$.   Let $V^k (\g)$  be  the associated universal affine vertex superalgebra, and  $L_k(\g)$ be its unique simple  $\dot{d}$--graded quotient.
As usual, we identify  $x \in {\g}$ with $x(-1) {\bf 1}$.

Define also the parafermion vertex algebra $N_{k}(\g) = \{ v \in L_k(\g) \ \vert \ h(n). v = 0 \ \forall n \in {\Z}_{\ge 0}\}$.
 
 The representation theory of the affine vertex superalgebra $L_k(\g)$  at admissible levels was recently studied in \cite{CFK, CKTR, RSW, KR1, W-2020}.
In the positive integer level case, $N_k(\g)$ is generated by $N_{k}(\g_0)$ and a  primary vector of weight three (cf. \cite{JW}).
  In the present paper  shall study the structure and  representation theory of $L_k(\g)$ and $N_k(\g)$ at the critical level $k=-3/2$.

\subsection{Realization}
We shall here recall some consequences of the  realization of the vertex superalgebra  $V^{-\tfrac{3}{2}}(\g )$ from \cite{A-2019}.

Consider the  lattice vertex algebra $\Pi(0)^{1/2} = M(1) \otimes {\C}[{\Z} \tfrac{c}{2}]$ as in \cite{A-2019} (see also \cite{BDT}, \cite{LW}).  
Let $g = \mbox{exp}[\pi i d(0)]$ be the automorphism of order two of $\Pi(0) ^{1/2}$.
 Set $\mu = \frac{d}{2} + \frac{k}{4} c, \nu = \frac{d}{2} - \frac{k}{4} c$.
For $\lambda \in {\C}$ and $r \in {\Z}$, we define
 $$ \Pi^{1/2} _{(r)} (\lambda):= \Pi^{1/2}(0). e^{r \mu + \lambda c}. $$
 Then   $\Pi^{1/2} _{(r)} (\lambda)$ is an untwisted ($g$--twisted)  $\Pi(0)^{1/2}$--module if $r$ is even (if $r$ is odd) (see \cite[Section 4.1]{A-2019}).

\begin{theorem} \cite{A-2019}
Let $k =-3/2$. There exists  a non-trivial homomorphism $$\Phi : V^{k}(\g) \rightarrow  \overline F \otimes \Pi(0)^{1/2}$$ such that 
 \bea\
e & \mapsto & e^{c }, \nonumber  \\
h & \mapsto & 2 \mu(-1), \label{def-h-3} \nonumber \\
f & \mapsto & \left[    (k+2)  \omega   -\nu(-1)^{2} -  (k+1) \nu(-2) \right] e^{-c} \nonumber \\
x & \mapsto &   \sqrt{2}   \Phi_1  (-\frac{1}{2})    e^{c/2 } \nonumber \\
y & \mapsto & \sqrt{2}   \left[  - \frac{i }{2} \Phi_2  (-\frac{3}{2})  + \Phi_1  (-\frac{1}{2})  \nu (-1) + \frac{2k+1}{2}   \Phi_1  (-\frac{3}{2})  \right] e^{-c/2} . \quad    \nonumber 
\eea 
 where $\omega$ is the Virasoro vector of $\overline F$ of central charge $c_{1,2}=-2$.
\end{theorem}

By using the boson-fermion correspondence (cf. \cite{K2}), one can derive expressions for odd generators $x,y$ in $ V_{\Z \alpha} \otimes \Pi(0)^{1/2}$:

\begin{lemma} \label{expr-xy}
We have:
\bea  x 
&=& e^{\alpha + c/2} + e^{-\alpha + c/2}. \label{expres-x} \\
  y &=&   
   \left[       \nu (-1) + \tfrac{1}{2}  \alpha(-1) \right]  e^{-\alpha  - c/2}
 +   \left[   \nu (-1) -\tfrac{3}{2}  \alpha(-1) \right]        e^{\alpha  - c/2}  \label{expres-y}
 \eea
  \end{lemma}  

 Let $\omega_{sug} ^{\g_0}$ be the Sugawara Virasoro vector in $V^k(\g_{0})$. Then 
$$ \omega_{sug} = \Phi(\omega_{sug} ^{\g_0} ) = \omega +\tfrac{1}{2} c(-1) d(-1) -\tfrac{1}{2} d(-2) + \frac{k}{4}  c(-2).$$
Set $L_{sug}  (n) =  (\omega_{sug})_{n+1}$. We have:
\bea
L_{sug} (-1) - D = Q. \label{vazna-rel-der}
\eea

\subsection{Screening operators} \label{screening-real}
Note that $\Pi_{\nu} := \Pi(0) ^{1/2} e^{\nu}$ is a twisted $\Pi(0) ^{1/2}$--module. Moreover $   F^{tw} \otimes  \Pi_{\nu}$ is a twisted $\overline F \otimes \Pi(0) ^{1/2}$--module. One shows   that $   F^{tw} \otimes  \Pi_{\nu}$ is an untwisted $V^{k}(\g)$--module.  Let $s = e^{\tfrac{\alpha}{2} + \nu}=v_{tw}  \otimes e^{\nu}$, and $S^{\mathfrak{osp} } = \int Y(s,z) dz = s_0$. 

\begin{lemma}  $S^{\mathfrak{osp}}$ is the  screening operator and commutes with the action of $\widehat{\g}$.
\end{lemma}
\begin{proof}
Note that $v_{tw}$ is a highest weight vector of conformal weight $-1/8$ which for $c_{1,2}=-2$ correspond to the singular vector denoted by  $v_{2,1}$ in \cite{A-2019}. Therefore $s= v_{2,1} \otimes e^{\nu}$ and $S^{\mathfrak{osp}}$ coincides with   the $\widehat{\g_0}$ screening operator from \cite{A-2019} and therefore it commutes with  the action of $\widehat{\g_0}$. It remains to prove that $S^{osp}$ commutes with operators  $x(n)$ and $y(n)$ for $n \in {\Z}$.

 By a direct calculation  and using expressions for $x,y$ from Lemma \ref{expr-xy} we get
 \bea  S^{\mathfrak{osp}} x =  S^{\mathfrak{osp}} y =  0. \label{vanishing-screening1} \eea
 Indeed, using standard calculation in lattice vertex algebras we get
 \bea S^{\mathfrak{osp}} x &=& e^{\alpha/2 + \nu} _0  (e^{\alpha + c/2} + e^{-\alpha + c/2}) \nonumber \\ & =& e^{\alpha/2 + \nu} _0  e^{-\alpha + c/2} = 0. \nonumber \\
 S^{\mathfrak{osp}} y &=& e^{\alpha/2 + \nu} _0  \left[       \nu (-1) + \tfrac{1}{2}  \alpha(-1) \right]  e^{-\alpha  - c/2} \nonumber \\
 && + e^{\alpha/2 + \nu} _0 \left[   \nu (-1) +k  \alpha(-1) \right]        e^{\alpha  - c/2} \nonumber  \\
 &=&  \left[       \nu (-1) + \tfrac{1}{2}  \alpha(-1) \right] e^{-\alpha /2 - c/2 + \nu}  \nonumber \\ 
 && -  e^{\alpha/2 + \nu} _{-1}  e^{-\alpha  - c/2}  = 0. \nonumber
 \eea

 Using the commutator formula,  we get
 $$ [S^{\mathfrak{osp}} , x(n)] = (S ^{\mathfrak{osp}} x)_n =0, [S^{\mathfrak{osp}} , y(n)] = (S^{\mathfrak{osp}} y)_n =0. $$
 The proof follows.
\end{proof}

Since the operator $S^{\mathfrak{osp}}$ coincides with the screening operator for $\widehat{\g_0}$ from  \cite{A-2019} and \cite{ACGY}, we    have the following consequence of \cite[Lemma 1]{ACGY}:
\begin{lemma} \label{lem-int} Let $w_n$ be the highest weight vector of $L^{Vir} (-2, \tfrac{n(n+1)}{2})$. Then
\bea L_{-\tfrac{3}{2}} ^{\g_0} (n \omega_1) &=& \mbox{Ker} _{ L^{Vir} (-2, \tfrac{n(n+1)}{2}) \otimes \Pi(0)^{1/2}} S^{\mathfrak{osp}} = L_{-3/2} (\g_0).  (w_n \otimes e^{\frac{n}{2} c}). \nonumber 
 \eea
\end{lemma}

\begin{theorem} \label{screening-realization} We have:
\item[(1)] $L_{-\tfrac{3}{2}} (\g) = \mbox{Ker}_{\overline F \otimes \Pi(0)^{1/2}} S^{\mathfrak{osp}}.  $
\item[(2)] As a $L_{-\tfrac{3}{2}} (\g_0)$--module:
$$L_{-\tfrac{3}{2}} (\g)  = \bigoplus_{n =0}  ^{\infty} (n+1) L_{-\tfrac{3}{2}} ^{\g_0} (n \omega_1), $$
where $\omega_1$ is the fundamental dominant weight of  $\g_0$.
\end{theorem}

\begin{proof} Assume that $w$ is a $\widehat{\g_0}$--singular vector in $\overline F \otimes \Pi(0)^{1/2}$ with dominant integral weight with respect to ${\g_0}$.  Using Lemma \ref{lem-int} we see that $w$ must have the form  $w_n \otimes e^{\frac{n}{2} c}$ for $n \in {\Bbb Z}_{ >0}$ such that $w_n$ is a singular vector for the Virasoro algebra with highest weight $\frac{n (n+1)}{2}$.  So $w_n =w_{n,j}$ for certain $ 0 \le j \le n$. Then using the same arguments as in \cite[Proposition 3]{ACGY} we get that  $\mbox{Ker}_{\overline F \otimes \Pi(0)^{1/2}} S^{\mathfrak{osp}}$ is simple. The assertion (1)  holds.

The proof of the decomposion in (2)  follows from  the screening realization in (1),  Lemma \ref{lem-int} and the decomposition of the vertex superalgebra $\overline F$ from Theorem \ref{sing-vekt-1}. Alternatively, the assertion follows directly using the decomposition of $\mathcal V^{(2)}$ as $L_{-3/2}(\g_0)$--module and Theorem \ref{general-vp-new} below. 
\end{proof}

\section{Coincidences and duality  between $L_{-3/2} (osp(1\vert 2))$ and   $\mathcal V^{(2)}$}
     
     The small  $N=4$ superconformal algebra is realized as the minimal affine $W$--algebra
     $\mathcal W_k(\mathfrak{psl}(2\vert 2), f_{\theta})$ (cf. \cite{KW}). 
     %
      %
      We shall denote $\mathcal W_{1/2}(\mathfrak{psl}(2\vert 2), f_{\theta})$ by $\mathcal V^{(2)}$, since it belongs to the  series of vertex algebras $\mathcal V^{(p)}$ defined in \cite{A-TG} and investigated in detail in \cite{ACGY}.

     Recall (cf.  \cite{A-2019}, \cite{ACGY}) that 
      $$ \mathcal V^{(2)}  = \mbox{Ker}_{ \mathcal{SF}(1)  \otimes \Pi(0) ^{1/2}}  S^{N=4}  $$
  where $$S^{N=4} = \int   e^{\frac{\alpha}{2} + \nu} (z)dz:  \mathcal{SF}(1)\otimes \Pi(0) ^{1/2}  \rightarrow 
     F_{tw} \otimes \Pi_{\nu}.$$
     Recall that there is a conformal embeddings $L_{-3/2}(\g_0) \hookrightarrow  \mathcal V^{(2)}$ and
     $$\mathcal V^{(2)} = \bigoplus_{n =0}  ^{\infty} (n+1) L_{-\tfrac{3}{2}} ^{\g_0} (n \omega_1). $$

 Then Theorem \ref{screening-realization}  implies that vertex superalgebras $\mathcal V^{(2)}$ and $L_{-3/2}(\g)$ are isomorphic as $L_{-3/2} (\g_0)$--modules.

     There are some other coincidences between $\mathcal V^{(2)}$ and 
     $L_{-3/2}(\g)$. By \cite{A-2019} we have screening operator 
     $$S  = \int   e^{\frac{\alpha}{2} + \nu} (z)dz :  F \otimes \Pi(0) ^{1/2}  \rightarrow 
     F_{tw} \otimes \Pi_{\nu}, $$
     such that
     $$ S \vert_{ \mathcal{SF}(1)\otimes \Pi(0)^{1/2} } \equiv S^{N=4}, \quad S \vert _{\overline F \otimes \Pi(0)^{1/2} } \equiv S^{\mathfrak{osp}}.  $$
     This shows that the $N=4$ superconformal algebra $\mathcal V^{(2)}$ and $L_{-3/2} (\g)$ are described as the kernels of the restrictions of the same screening operator. We have:
   $$ \mathcal V^{(2)} = \ \mbox{Ker}_{F \otimes  \Pi(0)^{1/2} } S \bigcap  \mbox{Ker}_{F \otimes  \Pi(0)^{1/2} } \Psi^- (1/2),  $$
     $$L_{-3/2}(\g)= \ \mbox{Ker}_{F \otimes  \Pi(0)^{1/2} } S \bigcap  \mbox{Ker}_{F \otimes  \Pi(0)^{1/2} } \Phi_2  (1/2). $$
     Thus, both algebras are intersection of kernels of two screening operators acting on $F \otimes  \Pi(0)^{1/2}$, the  screening $S$  is identical for both algebras, and only difference are the fermionic screenings acting on $F$.

Using  Proposition     \ref{isomorphism-fields-1}, we can prove a stronger result which says that  $L_{-3/2}(\g)$ can be equipped with the structure of a $\mathcal V^{(2)}$--module.
By Proposition     \ref{isomorphism-fields-1}(1),  the vertex superalgebra  $\widetilde{\mathcal{SF}}(1) \cong \mathcal{SF}(1)$ is realised as the  vertex superalgebra generated by local fields $\varphi^{\pm}(z)$ acting on $\overline{F}$. 
      
      Then 
we have the vertex superalgebra homomorphism 
$\mathcal V^{(2)} \rightarrow \widetilde{\mathcal{SF}}(1) \otimes \Pi(0)^{1/2}$. Denote by $\widetilde{\mathcal V}^{(2)}$ the image of this homomorphism. More precisely,
 since $\mathcal V^{(2)}$ is a simple vertex superalgebra, we have $\widetilde{\mathcal V}^{(2)} \cong \mathcal V^{(2)}$ as vertex superalgebras.

Applying the realisation from  \cite{A-TG} (see also  \cite{ACGY})  we get that $\widetilde{\mathcal V}^{(2)}$  is   generated by
 \begin{itemize}
 \item $\mathfrak{sl}(2)$ generators $e,f,h$ which generate $L_{-3/2}(\mathfrak{sl}(2))$,
 Sugawara Virasoro vector $\omega_{sug}$ (idenified with the  fields in  $\widetilde{\mathcal{SF}}(1) \otimes \Pi(0)^{1/2}$); 
 \item four odd primary fields $G^{\pm}(z), \overline G^{\pm}(z)$, which are expressed as
\bea &G^{+}(z) =  \varphi^{+}(z)   e^{\tfrac{c}{2}} (z) ,  &\overline G^+(z) = 2 \varphi^-(z)  e^{\tfrac{c}{2}}(z), \nonumber \\ &G^-(z) = f(0) G^+(z), \ &\overline G^-(z) =- f(0) \overline G^+(z). \nonumber \eea
 \end{itemize}

          \begin{theorem} \label{isom-osp-N4-direct}         
\item[(1)] $L_{-3/2}(\g)$  has the structure of an irreducible  $\mathcal V^{(2)}$--module, denoted by  $(L_{-3/2}(\g)  , \mathcal Y_{L_{-3/2}(\g)} ^{N=4} (\cdot, z))$, uniquely determined by
\bea  \mathcal Y_{L_{-3/2}(\g)} ^{N=4} ( v , z) &=& v(z)  \quad v \in L_{-3/2}(\g_0), \nonumber \\
 G^+(z) =   \mathcal Y_{L_{-3/2}(\g)} ^{N=4} ( G^+ , z) &=& x(z)  - z ( \frac{d}{dz} x(z) -  (L_{sug} (-1) x) (z)), \nonumber \\
  \overline G^+ (z) =  \mathcal Y_{L_{-3/2}(\g)} ^{N=4} ( \overline G^+ , z)  &=&  - \frac{d}{dz} x(z) +  (L_{sug} (-1) x) (z), \nonumber \\
   G^-(z) =   \mathcal Y_{L_{-3/2}(\g)} ^{N=4} ( G^- , z) &=& -y(z)  + z ( \frac{d}{dz} y(z) -  (L_{sug} (-1) y) (z)), \nonumber  \\
   \overline G^- (z) =  \mathcal Y_{L_{-3/2}(\g)} ^{N=4} ( \overline G^- , z)  &=& -  \frac{d}{dz} y(z) + (L_{sug} (-1) y) (z), \nonumber  
   \eea
  where for $v \in  L_{-3/2}(\g)$, we set $v(z) = Y_{ L_{-3/2}(\g) }(v,z)$.

\item[(2)] $\mathcal V^{(2)}$ has the structure of an (irreducible)   $L_{-3/2}(\g)$--module, 
 denoted by $(\mathcal V^{(2)}, \mathcal Y_{\mathcal V^{(2)}} ^{\g} (\cdot, z))$,which is uniquely determined by  
 \bea  \mathcal Y_{\mathcal V^{(2)}} ^{\g} ( v , z) &=& v(z), \quad v \in L_{-3/2}(\g_0), \nonumber \\
   x(z) = \mathcal Y_{\mathcal V^{(2)}} ^{\g} ( x , z) &=& G^+(z) - z \overline G^+(z), \nonumber \\
   y(z) = \mathcal Y_{\mathcal V^{(2)}} ^{\g} ( y , z) &=& -G^-(z) - z \overline G^+(z), \nonumber 
     \eea
 where for $v \in \mathcal V^{(2)} $, we set $v(z) =  Y_{\mathcal V^{(2)}}(v,z)$.
 \end{theorem}



 \begin{proof} 
   Recall the formula  (\ref{vazna-rel-der}) which gives
$L_{sug} (-1)   = D + Q $, where   $D$ is the derivation on the vertex superalgbra $L_{-3/2}(\g)$ and $Q$ is an operator on $\overline F$.  Therefore
  $$L_{sug}(-1) x = D x  + Q x= D x + \sqrt{2} (D \Phi_1 + \sqrt{-1} D\Phi_2)
  e^{\tfrac{c}{2}},  $$
  which implies that
  $$\sqrt{2} (D \Phi_1 + \sqrt{-1} D\Phi_2)   e^{\tfrac{c}{2}}= L_{sug}(-1) x - Dx. $$

  Now we apply formulas (\ref{embed-1})-(\ref{embed-2}) and get
  \bea G^+(z) &=&
    \left( z \sqrt{2} \left(  \partial_z \Phi_1(z)  +  \sqrt{-1} \partial_z \Phi_2(z) \right) + \sqrt{2} \Phi_1(z) \right) e^{\tfrac{c}{2}} (z) \nonumber  \\
    &=& x(z)  - z ( \frac{d}{dz} x(z) -  (L_{sug} (-1) x) (z))   \label{embed-n4-1}  \\
  \overline G^+(z) &=& -  ( \frac{d}{dz} x(z) -  (L_{sug} (-1) x) (z) ).   \label{embed-n4-2}   
   \eea
    Since $G^-(z), \overline G^-(z)$ are obtained from (\ref{embed-n4-1})-(\ref{embed-n4-2})  by applying the operator $f(0)$, we conclude that all  fields $G^{\pm}(z)$ and $\overline G^{\pm}(z)$ act on  $L_{-3/2}(\g)$. Therefore $L_{-3/2}(\g)$ is a $\mathcal V^{(2)}$--module.
 
Assume that  $L_{-3/2}(\g)$ is not irreducible $\mathcal V^{(2)}$--module. Then it has a proper submodule $0 \ne W \subsetneqq L_{-3/2}(\g)$.
By using formulas (\ref{embed-n4-1})-(\ref{embed-n4-2}) we get
\bea  x(z) =  G^+(z) - z \overline G^+(z), \label{embed-osp-x}\eea
which implies that $W$ is invariant for the field $x(z)$. Since $y(z)= -f(0) x(z)$, we get 
\bea  y(z) = - G^-(z) - z \overline G^-(z), \label{embed-osp-y}\eea
implying that $W$ is also invariant for $y(z)$. Therefore $W$ is $\hg$--invariant, which is a contradiction. This proves the assertion (1).
 
 The  proof of the assertion  (2) is  based on Proposition \ref{isomorphism-fields-1}(3), which shows that the vertex operator $\mathcal Y_{\mathcal{SF}(1)} ^{\overline F}$  defines on $\mathcal{SF}(1)\cong \widetilde{\mathcal{SF}}(1)$ the structure of an irreducible $\overline{F}$--module.
 
 Consider again $\mathcal V^{(2)}$ as a vertex subalgebra of $\widetilde{\mathcal{SF}}(1) \otimes \Pi(0)^{1/2}$. Then, as above, we get formulas (\ref{embed-osp-x})-(\ref{embed-osp-y}), which proves that $\mathcal V^{(2)}$ is a $L_{-3/2}(\g)$--modules.
 Since $G^+(z)$ and $\overline{G}^+(z)$ can be also  expressed from $x(z)$ using formulas  (\ref{embed-n4-1})-(\ref{embed-n4-2}) we conclude that $\mathcal V^{(2)}$ is an irreducible $L_{-3/2}(\g)$--module.
 \end{proof}
 
   \begin{remark}
     The $N=4$ superconformal vertex algebra $\mathcal V^{(2)}$ has appeared in the  $N=4$ super Yang-Mills theory in physics. In the  recent paper \cite{BN}, the authors found a very interesting  exact vector spaces isomorphism between  $\mathcal V^{(2)}$ and the doublet vertex algebra $\mathcal A^{(6)}$. The result from the present paper shows that there is another vector space isomorphism to the vertex algebra associated to $\mathfrak{osp}(1 \vert 2)$ at the critical level.
\end{remark}    
 \section{ A correspondence between $\mathcal V^{(2)}$ and $L_{-3/2}(\g)$--modules}
 
   
 Recent development in the representation theory of affine vertex algebras motivated the study of the category  $\mathcal R$  of modules   (see \cite[Section 2]{KR1} for  a formal definition) which includes:
 \begin{itemize}
 \item Ordinary modules (also called the category $KL_k$), 
 \item Highest weight and lowest weight modules, 
 \item Relaxed highest weight modules (\cite{A-TG}, \cite{A-2019}, \cite{KR1}).
 \end{itemize}
 
 Assume that $V= \mathcal V^{(2)}$ or $V= L_{-3/2}(\g)$. In our case one can show that  $V$--module $W$ is in the category $\mathcal R$ if and only if $W$ is $\tfrac{1}{2} {\Z}_{\ge 0}$--graded:
 $$ W= \bigoplus_{n  \in \tfrac{1}{2} {\Z}_{\ge 0} } W(n)$$
 and if $h(0)$ acts on each graded component $W(n)$ semi-simply with finite-dimensional weight components. 
 
 The irreducible modules in the category $\mathcal R$ can be obtained using Zhu's algebra theory. For any $A(\mathcal V^{(2)})$--module (resp.  $A(L_{-3/2}(\g))$--module) $U$,  let $L^{N=4} (U)$ (resp. $L_{-3/2} ^{\g} (U)$) denote the corresponding $\tfrac{1}{2} {\Z}_{\ge 0}$--graded vertex superalgebra module obtained using Zhu's theory. The irreducible $\tfrac{1}{2} {\Z}_{\ge 0}$--graded  $\mathcal V^{(2)}$--modules were classified in \cite{A-TG}. We will classify and construct  the irreducible $L_{-3/2}(\g)$--modules in the category $\mathcal R$ and find explicit correspondence between irreducible $\mathcal V^{(2)}$--modules and $L_{-3/2}(\g)$--modules.
 
 \label{corresp-modules}
 
\subsection{The representation theory of $\mathcal V^{(2)}$: revisited} 
 Let $U_{\mu, r}$, $\mu, r \in {\C}$,  be   the weight $\g_0= \mathfrak{sl}(2)$--module with basis: $E_i$, $i \in {\Z}$ and the $\g_0$--action is given by
   $$e E_i = E_{i-1}, \ h E_i = -(2r + 2i -\mu ) E_i, \ f E_i = - (r+ i +1)(r+i-\mu) E_{i+1}. $$
   The action of the Casimir $\Omega^{\g_0} = ef + fe + 1/2 h^2$ is
   $\Omega^{\g_0}  \vert  U_{\mu, r} = \frac{\mu (\mu +2)}{2} \mbox{Id}$.

The  representation theory of $\mathcal V^{(2)} = L^{N=4} _{c=-9}$ was studied in \cite{A-TG}. The vertex operator superalgebra 
   $\mathcal V^{(2)}$ has, up to a isomorphism or  parity reversing:
   \begin{itemize}
   \item[(1)] One irreducible module in the category of ordinary modules: the vertex operator superalgebra  $\mathcal V^{(2)}$ itself.
   \item[(2)]  Two irreducible modules in the category $\mathcal O$: $\mathcal V^{(2)}$ and  the highest weight module
   $L^{N=4}(U_{-1})$, where the top component is irreducible highest weight $\mathfrak{sl}(2)$--module with highest weight $-\omega_1$, where $\omega_1$ is the fundamental dominant weight of  $\mathfrak{sl}(2)$.
   \item[(3)]     
   The irreducible relaxed highest weight modules $L^{N=4}(U_{-1, r})$,  $r \notin {\Z}$, where the top component is isomorphic to  $U_{-1, r}$.  
   \end{itemize}
  These modules are explicitly realised in \cite{A-TG}.    Using embedding $\mathcal V^{(2)} \hookrightarrow F \otimes \Pi(0)^{1/2}$ we get a slightly reformulated result:
  
    \begin{prop} \cite{A-TG}  \label{N=4-rev}
    \item[(1)] There exist a non-split extension of  $\mathcal V^{(2)}$--modules
    $$ 0 \rightarrow \mathcal V^{(2)} \rightarrow \mathcal M_{-1} ^{N=4} \rightarrow L^{N=4} (U_{-1}) \rightarrow 0, $$
    where  $\mathcal M_{-1} ^{N=4}$ is a highest weight $\mathcal V^{(2)}$--module with highest weight vector $e^{\alpha -c/2}$.
    
    \item[(2)] Assume that $r \notin {\Z}$. We have:
  $$ L^{N=4}(U_{-1, r })=  \mathcal V^{(2)}.  e^{\alpha/2  -\mu - (r+1/4) c}.$$
  The character of $L^{N=4}(U_{-1, r })$ is given by
  \bea \mbox{ch}[L^{N=4}(U_{-1, r })](q,z) &=& \mbox{Tr} q^{L_{sug} (0)} z ^{h(0)}  \nonumber \\
  &=& z^{-2r} \delta(z^2) \frac{ \prod_{n=1} ^{\infty} (1+ q^{n-3/2} z^{-1})  \prod_{n=1} ^{\infty} (1+ q^{n+1/2} z^) } {\prod_{n=1} ^{\infty} (1-q^n) ^2 }. \nonumber 
  \eea
  \item[(3)] As a $\mathcal V^{(2)}$--modules we have
 $$ F^{tw} \otimes \Pi^{1/2} _{-1} (r+\tfrac{1}{4}) \cong L^{N=4}(U_{-1, r}) \oplus L^{N=4}(U_{-1, r+ \tfrac{1}{2} }) . $$
  \end{prop}
 \begin{proof}
 In \cite{A-TG} we proved that
 $$\mathcal M_{-1} ^{N=4}  = \mathcal V^{(2)}.  e^{-\delta}, \quad  L^{N=4}(U_{-1, r })=  \mathcal V^{(2)}.  e^{\beta -\delta - r (\alpha + \beta) },$$
 where $\delta, \alpha, \beta$ are as in \cite[Section 2]{ACGY} in the case $p=2$. Direct calculation shows that
 $$ -\delta = \alpha  - c/2, \beta -\delta - r (\alpha + \beta)= \alpha /2   -\mu - (r+1/4) c. $$
 
 Next we consider the twisted $F \otimes \Pi(0) ^{1/2}$--module
 $W  = F^{tw} \otimes \Pi_{-1} ^{1/2}(r+1/4) $ as un untwisted $\mathcal V^{(2)}$--module. The top component $W_{top}$ is isomorphic to the direct sum of $\g_0$--modules:
 $$ W_{top} = U_{-1,r}\oplus U_{-1, r+ \tfrac{1}{2}}. $$
 Then the representation theory of the vertex algebra $\mathcal V^{(2)}$ from \cite{A-TG} implies that
 $$ W \cong  L^{N=4}(U_{-1, r}) \oplus L^{N=4}(U_{-1, r+ \tfrac{1}{2} }). $$
 The proof follows.
 \end{proof}
 \subsection{The representation theory of $L_{-3/2}(\mathfrak{osp}(1 \vert 2))$}
 The universal affine vertex algebra $V^{-3/2} (\g)$ contains  a singular vector
 $$ T= \left( e(-1) f(-1) + f(-1) e(-1) + \frac{1}{2} h(-1) ^2 + \frac{1}{2} (x(-1) y(-1) - y(-1) x(-1)) \right) {\bf 1},$$
 therefore the Zhu's algebra $A(L_{-3/2}(\g))$ is a quotient of the associative algebra
 $U(\g) / \langle \Omega^{\g} \rangle$,
 where $\langle \Omega^{\g}\rangle$ is the two-sided ideal generated by the  Casimir central element 
$$  \Omega^{\g}= e f + f e + \frac{1}{2} h^2 + \frac{1}{2} (y x - x y ). $$
Let $\Sigma = x  y - y x + 1/2$ be the super Casimir.
%
We have (see \cite[Section 2]{W-2020}):
$$ \Omega^{\g}  = \frac{1}{2} \Sigma ^2 - \frac{1}{8}.$$

 Then on every  ${\Z}_{\ge 0}$--graded $L_{-3/2}(\g)$--module $W$ we must have:
 $$ \Omega^{\g}  \equiv 0 \quad \mbox{on} \ W_{top}. $$
 This implies that $ \Sigma ^2 = \frac{1}{4} \mbox{Id}$ and therefore:
 $$ W_{top} = W_{top} ^+ \oplus W_{top} ^- $$ such that
 $$ \Sigma \equiv \pm \frac{1}{2} \mbox{Id} \quad \mbox{on} \ W_{top}^{\pm}. $$ 
 We get:
 $$ \Omega^{\g_0}   \vert W_{top}^{+ } \equiv 0, \quad \Omega^{\g_0}  \vert W_{top}^{- } = -\frac{1}{2}\mbox{Id}. $$
  Since on $W_{top}$ we have $L_{sug} (0) \equiv \Omega^{\g_0}$, we get the following important lemma:
  
  \begin{lemma} \label{lemma-key} Let $W$ be a ${\Z}_{\ge 0}$--graded 
  $L_{-3/2}(\g)$--module. Then $W_{top} = W_{top} ^+ \oplus W_{top} ^-$ such that:
  $$ L_{sug} (0) \equiv 0 \ \mbox{on} \ W_{top} ^+, \quad  L_{sug}(0) \equiv -\frac{1}{2} \mbox{Id}  \ \mbox{on} \ W_{top} ^- . $$
  \end{lemma}
 
 Denote by $  U_{\g} (t\omega_1)$ the irreducible $\g$--module with highest weight $t\omega_1$, and the corresponding $V^{-3/2} (\g)$--module by $L_{-3/2} ^{\g} (t \omega_1)$.

 \begin{prop} \label{class-O}The set:
 $$ \{ L_{-3/2}(\g) , L_{-3/2} ^{\g} (  -\omega_1 )\}$$
 is a complete set of irreducible  $L_{-3/2}(\g)$--modules in the category $\mathcal O$. There exist  a highest weight module $\mathcal  M_{-1} ^{\g}$ such that the following extension of  $L_{-3/2}(\g)$--modules is non-split:
  $$ 0 \rightarrow L_{-3/2}(\g) \rightarrow 
\mathcal  M_{-1} ^{\g} 
 \rightarrow L_{-3/2} ^{\g} (  -\omega_1 )  \rightarrow 0. $$
 \end{prop}
 \begin{proof}
 Define $\mathcal   M_{-1} ^{\g}  = L_{-3/2}(\g). e^{\alpha - c/2}$.  
Then $\mathcal   M_{-1} ^{\g}$ is a highest weight $L_{-3/2}(\g)$--module, and therefore  its    simple quotient $L^{\g} _{-3/2} (   - \omega_1 )$ is also $L_{-3/2}(\g)$--module.

 Assume that $L^{\g} _{-3/2} (  t \omega_1)$ is a $L_{-3/2}(\g)$--module. Then 
 $$ \Omega^{\g}  \vert U_{\g} (t \omega_1) \equiv 0 \implies t (t+1) =0. $$
 Therefore $t = 0$ or $t = -1$.
 The case $t = 0$ corresponds to the vertex operator superalgebra $L_{-3/2}(\g)$, and  $t = -1$ to $L^{\g}_{-3/2} (  -\omega_1)$.

 Note that in $\mathcal  M_{-1} ^{\g}$ we have 
 $x(0) e^{\alpha - c/2} = {\bf 1}$, we have that $L_{-3/2}(\g)$ is a submodule of $\mathcal M_{-1} ^{\g}$. It remains to prove that
the quotient module $W= \mathcal  M_{-1} ^{\g} / L_{-3/2}(\g)$ is irreducible.

If  $W$ is not  simple, the $W$ must contain a proper submodule $W'$ which by Lemma \ref{lemma-key} should contain vectors of $L_{sug} (0)$--conformal weights $0$ or $-1/2$.

Since as a $L_{-3/2} (\g_0)$--modules  $$\mathcal  M_{-1} ^{\g} \cong \mathcal M_{-1} ^{N=4}, \mathcal V^{(2)} \cong L_{-3/2}(\g), $$
we conclude that $W'$ can not have vectors of $L_{sug}(0)$--conformal weights $0$ or $-1/2$. A contradiction. The proof follows.

 \end{proof}
 
 Now we want to construct and classify irreducible  relaxed highest weight $L_{-3/2}(\g)$--modules. The top components of these modules are irreducible $\g$--modules on which $\Omega^{\g}$ acts trivially. It is not difficult  to classify these  modules.
 
 \begin{lemma} Assume that $U$ is an irreducible  infinite-dimensional $\g$--module with $1$-dimensional weight spaces such that $\Omega^{\g} \vert U \equiv 0$. Then $U$ is isomorphic (up to parity reversing) to exactly one of the following modules \begin{itemize}
 \item highest weight module $U^{\g} (-\omega_1)$;
 \item lowest weight module $( U^{\g}(-\omega_1) )^{*}$;
 \item module $U^{\g} _0 (r)$ with basis $E_{i}, E_{i+1/2}$, $i \in {\Z}$ and $\g$--action defined by
 \bea
  && e E_{i} = E_{i-1}, \ h E_i = -(2 i + 2r +1) E_i,  \  f E_i = - (r+i+1) ^2 E_{i+1} \nonumber \\
 && e E_{i-1/2} = E_{i-3/2}, \ h E_{i-1/2} =- (2i + 2r) E_{i-1/2}, \ f E_{i-1/2} =- (r+i + 1) (r+i ) E_{i+1/2} \nonumber \\
 && x E_i = E_{i-1/2}, \ y E_i = - (r+i+1) E_{i+1/2}  \nonumber \\
 && x E_{i-1/2} = E_{i-1} , \ y E_{i-1/2} = - (r+i ) E_{i}. \nonumber 
 \eea
 As an $sl(2)$--module:  $U^{\g} _0 (r) \cong U_{-1, r} \oplus U_{0, r}$
 \end{itemize}
 
 \end{lemma} 
 \begin{proof} We already proved in Proposition \ref{class-O} that $U^{\g} (-\omega_1)$ is the unique irreducible  infinite-dimensional highest weight module annihilated by $\Omega^{\g}$. Arguments for lowest weight modules are completely dual.
 Next we assume that $U$ is neither lowest nor highest weight. Using Lemma \ref{lemma-key} we get that $U$ is as $\g_0$--module a direct sum of two weight modules with $1$--dimensional weight spaces on which $\Omega^{\g_0}$ acts by zero or $-\tfrac{1}{2} \mbox{Id}$. So:
 $$ U \cong U_{-1,r}  \oplus U_{0, s}.  $$
 for certain $r$ and $s$. But one gets that $U_{-1,r}  \oplus U_{0, s}$ is a $\g$--module if and only if $r\equiv s \ \mbox{mod}(\Z)$. The proof follows.
 \end{proof}
 
 Let $L^{\g} _{-3/2} (U_{0} (r))$ be the irreducible $V^{-3/2}(\g)$--module whose top component is  $U^{\g} _0 (r)$.
 \begin{theorem} Assume that $r \notin {\Z}$. Then 
 $L^{\g} _{-3/2} (U_{0} (r))$ is a  $L_{-3/2}(\g)$--module and it is realised as
 $$ L^{\g} _{-3/2} (U_{0} (r)) = L_{-3/2} (\g). e^{\alpha/ 2 - \mu - (r + 1/4) c}. $$
 The basis of the top component $U^{\g} _{0} (r)$ is given by
 $$ E_i = e^{\alpha/ 2 - \mu -  (r +  i + 5/4) c}, \quad
  E_{i-1/2} = e^{-\alpha/ 2 - \mu -  (r +  i + 3/4) c}. $$
  As a $L_{-3/2}(\g)$--module we have:
  $$ F_{tw} \otimes  \Pi_{-1} ^{1/2} (r+\tfrac{1}{4}) \cong  L^{\g}_{-3/2} (U_{0} (r)) \oplus  L^{\g} _{-3/2} (U_{0} (r+\tfrac{1}{2})). $$
 \end{theorem}
 \begin{proof} Let $W = L_{-3/2}(\g). e^{\alpha/ 2 - \mu -  (r +  i + 5/4) c}$.
   By direct calculation we get:
 \bea
   &e(0)  E_{i} = E_{i-1},  &e(0) E_{i-1/2} = E_{i-3/2} \nonumber \\
    &h(0)  E_i = -(2 i + 2r +1) E_i,    &h(0)  E_{i-1/2} =- (2i + 2r ) E_{i-1/2} \nonumber \\
    &f(0)  E_i = - (r+i+1) ^2 E_{i+1},  &f(0) E_{i-1/2} =- (r+i + 1) (r+i) E_{i+1/2} \nonumber \\
  &x(0)  E_i = E_{i-1/2},  &y(0)  E_i = - (r+i+1) E_{i+1/2}  \nonumber \\
  &x(0)  E_{i-1/2} = E_{i-1} ,  &y(0)  E_{i-1/2} = - (r+i ) E_{i} \nonumber 
 \eea
 So $W$ is a cyclic ${\Z}_{\ge 0}$--graded  $L_{-3/2}(\g)$--module, whose top component  $W_{top}$ is isomorphic to the irreducible $\g$--module $U^{\g} _{0} (r)$.  If $W$ is not irreducible, then there is a proper submodule $Z \subsetneqq W$ which intersects $W_{top}$ trivially. By using Lemma \ref{lemma-key}, we conclude that $Z$ has vectors of $L_{sug}(0)$--conformal weights $0$ or $-1/2$. But since $W \cong L^{N=4}(U_{-1,r})$ as $\widehat{\g_0} $--module, using character formula for $L^{N=4}(U_{-1,r})$ (cf. Proposition \ref{N=4-rev}) we see that  all vectors of conformal weights $0$ or $-1/2$ should correspond to $ U_{-1,r} \oplus U_{0,r} \cong W_{top}$ which intersect  $Z$ trivially. A contradiction. Therefore $W$ is irreducible.
 \end{proof}

 \section{Generalization to logarithmic vertex algebras}
\label{generalizacija}
 
 In this section we will see that correspondences
 $$ \mathcal{SF}(1)	\longleftrightarrow \overline F, \quad \mathcal V^{(2)} 	\longleftrightarrow  L_{-3/2} (osp(1 \vert 2)) $$
 can be extended to a larger family of logarithmic vertex algebras.
 \vskip 5mm

 Consider the generalized lattice vertex algebra $V_{\widetilde L}$ associated to the lattice
 
 $$ \widetilde L = {\Z} \gamma  /2, \quad \langle \gamma, \gamma \rangle = 2p . $$
 Recall \cite{AM-doublet} that the doublet vertex algebra is defined as:
   $$\mathcal A^{(p)} = \mbox{Ker}_{V_{\widetilde L} } \widetilde Q, $$
   where $$\widetilde Q = e^{-\gamma/p}_0 = \int e^{-\gamma/p} (z) dz. $$
   Doublet algebra  has the Virasoro vector
   $$ \omega^{(p)} = \frac{1}{4p} \gamma(-1) ^2 +\frac{p-1}{2p} \gamma(-2)$$
   of central charge $c_{p,1} = 1- \frac{6 (p-1)^2}{p}$ and derivation
   $Q = e^{\gamma}_0 = \int e^{\gamma} (z) dz$. Let
   $$ L_{st} (z) = Y( \omega^{(p)}, z) = \sum_{n \in {\Z}} L_{st} (n) z^{-n-2}. $$

 Define a new Virasoro vector of central charge $c_{p,1}$:
 $$\omega= \omega^{(p)} _{new}= \omega^{(p)} +   e^{\gamma} = \frac{1}{4p} \gamma(-1) ^2 +\frac{p-1}{2p} \gamma(-2) + e^{\gamma}. $$
 Let $ L(n)  = \omega_{n+1}$, 
 and
 \bea v = e^{-\gamma /p}  - \frac{1}{p-1}  e^{\gamma-\gamma/p} \label{screening-new-1} \eea
 Then we have:
 \bea  L  (0) v &=& e^{-\gamma  /p} + e^{\gamma - \gamma /p} \nonumber \\
 L   (-1) v & =&  D v  + \gamma(-1) e^{\gamma -\gamma /p}  =  D v + \frac{p}{p-1} D e^{\gamma-\gamma/p} \nonumber \\ &=& D (e^{-\gamma /p} +  e^{\gamma - \gamma /p}) \nonumber
 \eea
 This implies that
  $$\widetilde Q _{new} = e^{-\gamma /p} _0  - \frac{1}{p-1} e^{\gamma-\gamma/p} _0 $$
 is a screening operator.
 Define new generalized vertex algebra
 $$\mathcal A^{(p)} _{new} = \mbox{Ker}_{V_{\widetilde L} } \widetilde Q _{new}. $$
 
 In the case $p=2$, we get $\mathcal A^{(p)} = \mathcal{SF}(1)$ and $\mathcal A^{(p)} _{new} = \overline F$ and we proved that $\mathcal{SF}(1)$  and  $\overline F$  are isomorphic as $L^{Vir}(-2,0)$--modules.  
 \begin{theorem} \label{conj-1}$\mathcal A^{(p)}$ and $\mathcal A^{(p)} _{new}$ are isomorphic as $L(c_{p,1}, 0)$--modules and
 \bea \mathcal A^{(p)} _{new} = \bigoplus_{n =0} ^{\infty} (n+1) L(c_{p,1}, \frac{n (np + 2p-2))}{4}).\label{dec-ap-nova-2} \eea The isomorphism is given by
 $$ \Omega\vert _{\mathcal A^{(p)}} : \mathcal A^{(p)} \rightarrow \mathcal A^{(p)} _{new}$$
 where $\Omega$ as an operator $V_{\widetilde L}$ given by 
  $$ \Omega = \exp[e^{\gamma} _1] = \sum_{n = 0} ^{\infty} \frac{ (e^{\gamma} _1) ^n}{n!}. $$
 \end{theorem}
 \begin{proof}
 Note that $\Omega$ is invertible and  $\Omega^{-1} =\exp[-e^{\gamma} _1]$.
 Since
 $$ [\widetilde Q, e^{\gamma} _1] = -(\tfrac{\gamma (-1)}{p} e^{\gamma - \tfrac{\gamma}{p}})_1 =-\tfrac{1}{p-1} (De^{\gamma - \tfrac{\gamma}{p}})_1 =\tfrac{1}{p-1}  e^{\gamma - \tfrac{\gamma}{p}} _0, $$
 we get for $n \ge 1$:
 $$ \frac{ (e^{\gamma} _1 )^n }{n!} \widetilde Q = \widetilde Q \frac{ (e^{\gamma} )_1 ^n }{n!}  -\frac{1}{p-1}  e^{\gamma - \tfrac{\gamma}{p}} _0 \frac{(  e^{\gamma} _1 ) ^{n-1} }{(n-1)!}. $$
This implies
\bea  \Omega \widetilde Q &=& \widetilde Q_{new} \Omega \nonumber \\
 \Omega^{-1}  \widetilde Q_{new} &=& \widetilde Q  \Omega^{-1} \nonumber \eea
 Since
  $$L  (n) \Omega = \Omega L_{st} (n), $$
  we get that 
  $$ \Omega \vert _{\mathcal A^{(p)} } : \mathcal A^{(p)} \rightarrow \mathcal A^{(p)} _{new}$$
  is an isomorphism of $L^{Vir}(c_{p,1}, 0)$--modules. The proof of the decomposition (\ref{dec-ap-nova-2}) follows from the decomposition of  $\mathcal A^{(p)}$ as a direct sum of  irreducible $L^{Vir}(c_{p,1}, 0)$--modules (cf. \cite{AM08}, \cite{AM-doublet}, \cite{ACGY}).
 \end{proof}

Recall that the $\Z_2$--orbifold of $\mathcal A^{(p)}$ is the triplet vertex algebra $\mathcal W^{(p)} = \mbox{Ker}_{ V_{\Z \gamma}} \widetilde Q$ (cf. \cite{AM08}) generated by
$$ \omega^{(p)}, \  F= e^{-\gamma}, \  H  = Q F, \  E = Q^2 F. $$
Define:
$$\mathcal W^{(p)} _{new}  = \mbox{Ker}_{ V_{\Z \gamma}} \widetilde Q_{new}. $$

\begin{theorem} \label{generators-general} 
$\mathcal W^{(p)} _{new} $ is generated by  
$$ \omega^{(p)}_{new}, F_{new} = \Omega F,  H_{new} = \Omega H, E_{new} =\Omega E. $$
\end{theorem}

\begin{proof}
The proof is similar to that of \cite[Proposition 1.3]{AM08}.  We know that
$\mathcal W^{(p)} _{new}$ is as a module for the Virasoro algebra isomorphic to $\mathcal W^{(p)}$ and it is generated by the following singular vectors:
$$ \Omega Q^j e^{-n \gamma} =   Q^j e^{-n \gamma}  +  Q^j   e^{\gamma}_1 e^{-n \gamma} + \cdots ,$$
for  $n \in {\Z}_{\ge 0}$, $0 \le j \le 2n$.
This implies that
$$  \Omega Q^j e^{-n \gamma} = Q^j e^{-n \gamma} + z'_{j},   \quad z'_{j} \in V_{\Z \gamma}, \ Q^{2n-j} z'_j = 0. $$
Let $Z_n$ be the Virasoro module generated by  singular vectors
$$\Omega Q^j e^{-m \gamma},  \ m \le n, j \ge 0 . $$
Then
$$W^{(p)} _{new} = \bigcup_{n \in {\Z} } Z_n. $$
Let $U$ be the vertex subalgebra of $\mathcal W^{(p)} _{new}$ generated by $F_{new}, H_{new}, E_{new}, \omega^{(p)} _{new}$.  

We will   show by induction  that $Z_n \subset U$ for every $n \in {\Z}_{>0}$. For $n=1$, the claim holds. Assume that $Z_n \subset U$. Set $\ell = -2 n p -1$.    We use the following relations in $\mathcal W^{(p)}$  proved in  \cite[Proposition 1.3]{AM08}:
$$ F_{\ell} e^{-n \gamma} = \nu_1  e^{-(n+1) \gamma},\quad  E_{\ell} Q^{2n}  e^{-n \gamma}  = \nu_2 Q^{2n+2}  e^{-(n+1) \gamma},$$
$$  H_{\ell}    Q^j  e^{-n\gamma} = C_j   Q^{j+1} e^{-(n+1) \gamma} +v'_{j}, \ \   v'_j \in \mbox{Ker}_{V_{\Z\gamma}} Q^{2n+1-j},$$
$j=1, \cdots, 2n$, $\nu_1, \nu_2, C_j \ne 0$. These relations imply that in $\mathcal W^{(p)} _{new}$ we have:
\bea  (F_{new} )_{\ell}     \Omega e^{-n \gamma} &=& \nu_1  \Omega   e^{-(n+1) \gamma} + v_0, 
 v_0 \in \mbox{Ker} _{V_{\Z \gamma}} Q^{2n+2} \nonumber \\
 (E_{new} )_{\ell}  \Omega Q^{2n} e^{-n \gamma}   &=& \nu_2 \Omega Q^{2n+2}  e^{-(n+1) \gamma}, \nonumber \\
   (H_{new} )_{\ell}   \Omega Q^j  e^{-n\alpha} &=& C_j  \Omega Q^{j+1} e^{-(n+1) \gamma} +v_{j}, \ \   v_j \in \mbox{Ker}_{V_{\Z \gamma}}  Q^{2n+1-j}.\nonumber \eea
 We conclude that $\Omega Q^{j} e^{-(n+1) \gamma} \in Z_n$ for $j=0, \dots, 2n+2$. These relations imply that $Z_{n+1} \in U$. The claim now follows by induction. Therefore $U = \mathcal W^{(p)}_{new}$.
\end{proof}

\vskip 5mm
 Next we consider the vertex algebra $\mathcal A^{(p)} _{new} \otimes \Pi(0)^{1/2}$. 
 Then we have  the vertex algebra homomorphism  $L_{-2+ \frac{1}{p}} (\mathfrak{sl}(2))  \rightarrow \mathcal A^{(p)} _{new} \otimes \Pi(0)^{1/2}$ with screening operator
 $ S^{(p)}= s^{(p)}_0$, where $s^{(p)} = e^{\frac{\gamma}{2p} + \nu}$.
 We define:
 $$ \mathcal V^{(p)}_{new} = \mbox{Ker}_{ \mathcal A^{(p)} _{new} \otimes \Pi(0)^{1/2}   } S^{(p)}.$$
 In the case $p=2$ we already proved that $ \mathcal V^{(2)}_{new} = L_{-3/2} (\mathfrak{osp}(1 \vert 2))$ and that $L_{-3/2} (\mathfrak{osp}(1 \vert 2))$ is isomorphic to $\mathcal V^{(2)}$ as $L_{-3/2} (\mathfrak{sl}(2))$--module. Next results extends this for $p >2$.
 
 \begin{theorem}    \label{general-vp-new}   Assume that $p \ge 2$. We have:
 $$\overline \Omega= \Omega \otimes \mbox{Id} \ \vert _{\mathcal V^{(p)} } : 
\mathcal V^{(p)} \rightarrow \mathcal V^{(p)}_{new}$$
is an isomorphism of $L_{-2 + \tfrac{1}{p}}(\mathfrak{sl}(2))$--modules. 
 \end{theorem} 
 \begin{proof}
 First we notice that  $\overline \Omega$ commutes with the action of $ S^{(p)}$ and therefore it defines a linear bijection $\mathcal V^{(p)} \rightarrow \mathcal V^{(p)}_{new}$. The claim now follows from the relation:
 $$ x_{new} (n) \overline \Omega = \overline \Omega x  (n) \quad  x\in \mathfrak{sl}(2).  $$
 \end{proof}

\section{The structure of the parafermion vertex algebra $N_{-\tfrac{3}{2}} (\mathfrak{osp}(1 \vert 2))$ }

 First consider the parafermion vertex algebras  of $\mathcal V^{(p)}$ and $\mathcal V^{(p)}_{new}$. Let $M_h(1)$ be the Heisenberg vertex algebra generated by $h$. Let
\bea \mathcal N^{(p)} &=& \mbox{Com}(M_h(1), \mathcal V^{(p)} ) = \{ v \in  \mathcal V^{(p)} \vert h(n) v = 0, n \ge 0 \}, \nonumber \\  \mathcal N^{(p)} _{new} &=&  \mbox{Com}(M_h(1),  \mathcal V^{(p)} _{new} ) =
  \{ v \in  \mathcal V^{(p)} _{new} \vert h(n) v = 0, n \ge 0 \}. \nonumber \eea
 
 Since the operator $ \overline \Omega =\Omega \otimes \mbox{Id}$ commutes with operators $h(n)$, $n \in {\Z}$,  Theorem \ref{general-vp-new} directly implies: 
 \begin{cor}    \label{general-vp-new-paraferm}   We have:
 $$\overline \Omega   \vert _{\mathcal N^{(p)} } : 
\mathcal N^{(p)} \rightarrow \mathcal N^{(p)}_{new}$$
is an isomorphism of $N_{-2 + \tfrac{1}{p}}(\g_0)$--modules.  In particular, as $N_{-2 + \tfrac{1}{p}}(\g_0)$--modules we have
$$\mathcal N^{(p)}   \cong \mathcal N^{(p)}_{new} \cong  \bigoplus_{n =0}  ^{\infty} (2 n+1) N_{-2 + \frac{1}{p}} ^{\g_0} (2 n \omega_1). $$
 \end{cor}
 \vskip 5mm
 The most interesting case is $p=2$, since then $\mathcal N^{(2)} _{new} = N_{-\tfrac{3}{2}} (\g)$.    
Let us now determine the generators of $N_{-\tfrac{3}{2}} (\g)$.

 If $U_1$ and $U_2$ are vector subspaces of the vertex algebra $V$, denote by  
$$ U_1 \cdot U_2 = \mbox{span}_{\C} \{ u _n v \ \vert u \in U_1 , v \in U_2\}$$   the fusion product of $U_1$ and $U_2$. If $U_1, U_2$ are modules for a vertex subalgebra $V_0$ of $V$, then $U_1 \cdot U_2$ is also a $V_0$--module.

The next  lemma  follows from the proof of Theorem  \ref{conj-1} in the case $p=2$.
\begin{lemma} \label{gen-p2}
\item[(1)] $\mathcal W^{(2)}_{new}= \overline F^+$ is a simple vertex algebra strongly generated by
$\omega,   w_{2,2},  w_{2,1}, w_{2,0}.$
\item[(2)] Let $U_{n,j}$ be the Virasoro module generated by $w_{n,j}$. Then we have:
$$  U_{2n+2, 2n+2} \subset U_{2,2} \cdot U_{2n,2n}, \quad U_{2n+2, 0} \subset U_{2,0}  \cdot U_{2n,0}, \quad U_{2n+2, j+1} \subset  U_{2,1} \cdot  U_{2n, j }
$$
where $j=0, \dots, 2 n$. 
 \end{lemma}

\vskip 5mm

\begin{theorem}
  $N_{-3/2}(\g)$ is generated by $N_{-3/2}(\g_0)$ and three primary vectors $Z_{2,0}, Z_{2,1}, Z_{2,2}$  of conformal weight $4$.
\end{theorem}
\begin{proof} Let $W_{n,j}:= L_{-3/2} (\g_0). (w_{n,j} \otimes e^{\frac{n}{2} c}) \cong L_{-3/2} ^{\g_0}  (n \omega_1)$. Using Lemma  \ref{gen-p2} we get the following fusion rules between $L_{-3/2}(\g_0)$--submodules of $L_{-3/2}(\g)$:
\bea  &&W_{2n+2, 2n+2} \subset W_{2,2} \cdot W_{2n,2n}, \quad W_{2n+2, 0} \subset   W_{2,0}    \cdot W_{2n,0}, \quad W_{2n+2, j+1} \subset   W_{2,1} \cdot  W_{2n, j },
\label{affine-fusion}\eea
where $j=0, \dots, 2n$.   
Let $N_{n,j} = \{ v \in W_{n,j} \ \vert \ h(n) v = 0 \ \forall n \ge 0\}$.
 Note that
 $N_{n,j} = 0$ if $n$ is odd. 
 Using restriction  of the fusion rules  (\ref{affine-fusion}) to the parafermion algebra, we get the following fusion rules between $N_{-3/2} (\g_0)$--modules inside of $N_{-3/2}(\g)$:
 \bea  &&N_{2n+2, 2n+2} \subset N_{2,2} \cdot N_{2n,2n}, \quad N_{2n+2, 0} \subset   N_{2,0}    \cdot N_{2n,0}, \quad N_{2n+2, j+1} \subset   N_{2,1} \cdot  N_{2n, j },
\label{parafermion-fusion}\eea
where $j=0, \dots, 2n$.
 
 Using  Corollary   \ref{general-vp-new-paraferm} we get the following decomposition:
 $$ N_{-\tfrac{3}{2}}(\g) = \bigoplus_{n=0} ^{\infty} \bigoplus_{j=0}^{2n} N_{2n,j}. $$
 Now relation (\ref{parafermion-fusion}) easily implies that $N_{-\tfrac{3}{2}}(\g)$ is generated by $N_{-\tfrac{3}{2}}(\g_0)$ and $N_{2,j}$, $j=0,1,2$.
Moreover, $N_{2n,j}$ is an irreducible $N_{-3/2} (\g_0)$--module generated by  a highest weight vector  which we denote by $Z_{2n,j}$. Thus, $N_{-\tfrac{3}{2}}(\g_0)$ is generated by $N_{-\tfrac{3}{2}}(\g_0)$ and  three highest weight vectors $Z_{2,j}$, $j=0,1,2$.
  \end{proof}
 
 \section*{Acknowledgment}
 
D.A. is  partially supported   by the
QuantiXLie Centre of Excellence, a project coffinanced
by the Croatian Government and European Union
through the European Regional Development Fund - the
Competitiveness and Cohesion Operational Programme
(KK.01.1.1.01.0004). 
Q. W. is supported by China NSF grants No.12071385 and the Fundamental Research Funds for the Central
Universities No.20720200067

\vskip10pt {\footnotesize{}{ }\textbf{\footnotesize{}D.A.}{\footnotesize{}:
Department of Mathematics, University of Zagreb, Bijeni\v{c}ka 30,
10 000 Zagreb, Croatia; }\texttt{\footnotesize{}adamovic@math.hr}{\footnotesize \par}

\vskip 10pt \textbf{\footnotesize{}Q.W.}{\footnotesize{} School of Mathematical
Sciences, Xiamen University, Fujian, 361005, China;} \texttt{\footnotesize{}
qingwang@xmu.edu.cn}{\footnotesize \par}

\end{document}